\documentclass[12pt]{amsart}
\usepackage{amssymb}
\usepackage{hyperref}

\newtheorem{theorem}{Theorem}[section]
\newtheorem{definition}[theorem]{Definition}
\newtheorem{proposition}[theorem]{Proposition}

\newtheorem{lemma}[theorem]{Lemma}
\newtheorem{corollary}[theorem]{Corollary}

\begin{document}

\title{Submatrices of Hadamard matrices: complementation results}

\author{Teo Banica}
\address{T.B.: Department of Mathematics, Cergy-Pontoise University, 95000 Cergy-Pontoise, France. {\tt teo.banica@gmail.com}}

\author{Ion Nechita}
\address{I.N.: CNRS, Laboratoire de Physique Th\'eorique, IRSAMC, Universit\'e de Toulouse, UPS, 31062 Toulouse, France. {\tt nechita@irsamc.ups-tlse.fr}}

\author{Jean-Marc Schlenker}
\address{J.-M.S.: University of Luxembourg, Campus Kirchberg
Mathematics Research Unit, BLG
6, rue Richard Coudenhove-Kalergi
L-1359 Luxembourg. \tt{jean-marc.schlenker@uni.lu}}


\begin{abstract}
Two submatrices $A,D$ of a Hadamard matrix $H$ are called {\em complementary} if, up to a permutation of rows and columns, $H=[^A_C{\ }^B_D]$. We find here an explicit formula for the polar decomposition of $D$. As an application, we show that under suitable smallness assumptions on the size of $A$, the complementary matrix $D$ is an {\em almost Hadamard sign pattern}, i.e. its rescaled polar part is an almost Hadamard matrix.
\end{abstract}

\maketitle

\tableofcontents

\section*{Introduction}

A Hadamard matrix is a square matrix $H\in M_N(\pm1)$, whose rows are pairwise orthogonal. The basic example is the Walsh matrix, having size $N=2^n$:
$$W_N=\begin{bmatrix}+&+\\+&-\end{bmatrix}^{\otimes n}$$

In general, the Hadamard matrices can be regarded as ``generalizations'' of the Walsh matrices. Their applications, mostly to coding theory and to various engineering questions, parallel the applications of the Walsh functions and matrices.

Mathematically speaking, $W_N=W_2^{\otimes n}$ is the matrix of the Fourier transform over the group $G=\mathbb Z_2^n$, and so the whole field can be regarded as a ``non-standard'' branch of discrete Fourier analysis. Of particular interest here is the Hadamard conjecture: for any $N\in 4\mathbb N$, there exists a Hadamard matrix $H\in M_N(\pm1)$. See \cite{sya}, \cite{sha}.

We are interested here in square submatrices of such matrices. Up to a permutation of rows and columns we can assume that our submatrix appears at top left or bottom right:

\medskip

\noindent {\bf Setup.} {\em We consider Hadamard matrices $H\in M_N(\pm1)$ written as
$$H=\begin{bmatrix}A&B\\C&D\end{bmatrix}$$
with $A\in M_r(\pm1)$ and $D\in M_d(\pm1)$, where $N=r+d$.}

\medskip

As a first observation, one can show that any $\pm1$ matrix appears as submatrix of a certain large Walsh matrix, so nothing special can be said about $A,D$. That is, when regarded individually, $A,D$ are just some ``random'' $\pm1$ matrices.

The meaning of the word ``random'' here is in fact quite tricky. In general, the random Bernoulli matrices $D\in M_d(\pm1)$ are known to obey to the Tao-Vu rules \cite{tv1}, \cite{tv2}, and their refinements, and basically to nothing more, in the $d\to\infty$ limit.

For submatrices of Hadamard matrices, however, the situation is much more complicated, and what ``random'' should really mean is not clear at all. All this is of course related to the Hadamard Conjecture. See de Launey and Levin \cite{ll1}, \cite{ll2}.

Now back to our matrices $A,D$, the point is to consider them ``together''. As a first remark here, the unitarity of $U=\frac{H}{\sqrt{N}}$ gives, as noted by Sz\"{o}ll\H{o}si in \cite{szo}:

\medskip

\noindent {\bf Fact.} {\em If $A\in M_r(\pm1),D\in M_d(\pm1)$ are as above then the singular values of $\frac{A}{\sqrt{N}},\frac{D}{\sqrt{N}}$ are identical, up to $|r-d|$ values of $1$. In particular, $|\det\frac{A}{\sqrt{N}}|=|\det\frac{D}{\sqrt{N}}|$.}

\medskip

This simple fact brings a whole new perspective on the problem: we should call $A,D$ ``complementary'', and see if there are further formulae relating them.

Let us recall now a few findings from \cite{bcs}, \cite{bne}, \cite{bns}, \cite{bnz}. As noted in \cite{bcs}, by Cauchy-Schwarz an orthogonal matrix $U\in O(N)$ satisfies $||U||_1\leq N\sqrt{N}$, with equality if and only if $H=\sqrt{N}U$ is Hadamard. This is quite nice, and leads to:

\medskip

\noindent {\bf Definition.} {\em A square matrix $H\in M_N(\mathbb R)$ is called almost Hadamard (AHM) if the following equivalent conditions are satisfied:
\begin{enumerate}
\item $U=H/\sqrt{N}$ is orthogonal, and locally maximizes the $1$-norm on $O(N)$. 

\item $U_{ij}\neq 0$ for any $i,j$, and $US^t>0$, where $S_{ij}=sgn(U_{ij})$.
\end{enumerate}
In this case, we say that $S\in M_N(\pm1)$ is an almost Hadamard sign pattern (AHP).}

\medskip

In this definition the equivalence $(1)\iff(2)$ follows from a differential geometry computation, performed in \cite{bcs}. For results on these matrices, see \cite{bcs}, \cite{bne}, \cite{bns}, \cite{bnz}.

For the purposes of this paper, observe that we have a bijection as follows, implemented by $S_{ij}=sgn(H_{ij})$ in one direction, and by $H=\sqrt{N}Pol(S)$ in the other:
$$AHM\longleftrightarrow AHP$$

With these notions in hand, let us go back to the matrices $A\in M_r(\pm1),D\in M_d(\pm1)$ above. It was already pointed out in \cite{bnz}, or rather visible from the design computations performed there, that when $r=1$, the matrix $D$ must be AHP. We will show here that this kind of phenomenon holds under much more general assumptions. First, we have:

\medskip

\noindent {\bf Lemma.} {\em Let $H=\begin{bmatrix}A&B\\C&D\end{bmatrix}\in M_N(\pm1)$ be a Hadamard matrix, such that $A\in M_r(\pm 1)$ is invertible. Then, the polar decomposition $D=UT$ is given by
$$U=\frac{1}{\sqrt{N}}(D-E)\qquad T=\sqrt{N}I_d-S$$ 
where $E=C(\sqrt{N}I_r+\sqrt{A^tA})^{-1}Pol(A)^tB$ and $S=B^t(\sqrt{N}I_r+\sqrt{AA^t})^{-1}B$.}

\medskip

These formulae, which partly extend the work in \cite{km1}, \cite{km2}, \cite{szo}, will allow us to estimate the quantity $||E||_\infty$, and then to prove the following result:

\medskip

\noindent {\bf Theorem.} {\em Given a Hadamard matrix $H=\begin{bmatrix}A&B\\C&D\end{bmatrix}\in M_N(\pm1)$ with $A\in M_r(\pm1)$, $D$ is an almost Hadamard sign pattern (AHP) if:
\begin{enumerate}
\item $A$ is invertible, and $r=1,2,3$.

\item $A$ is Hadamard, and $N>r(r-1)^2$.

\item $A$ is invertible, and $N>\frac{r^2}{4}(r+\sqrt{r^2+8})^2$
\end{enumerate}}

\medskip

The paper is organized as follows: \ref{sec:AHM}-\ref{sec:submatrices} are preliminary sections, in \ref{sec:polar} we state and prove our main results, and in \ref{sec:small-r}-\ref{sec:non-AHP} we discuss examples, and present some further results.

\bigskip
 
\noindent\textbf{Acknowledgements.} The work of I.N. was supported by the ANR grants ``OSQPI'' {2011 BS01
  008 01} and ``RMTQIT''  {ANR-12-IS01-0001-01}, and by the PEPS-ICQ CNRS project ``Cogit''.

\section{Almost Hadamard matrices}\label{sec:AHM}

A Hadamard matrix is a square matrix $H\in M_N(\pm1)$, whose rows are pairwise orthogonal. By looking at the first 3 rows, we see that the size of such a matrix satisfies $N\in\{1,2\}\cup 4\mathbb N$. In what follows we assume $N\geq 4$, so that $N\in4\mathbb N$.

We are interested here in the submatrices of such matrices. In this section we recall some needed preliminary material, namely: (1) the polar decomposition, and (2) the almost Hadamard matrices. Let us begin with the polar decomposition \cite{hjo}:

\begin{proposition}
Any matrix $D\in M_N(\mathbb R)$ can be written as $D=UT$, with positive semi-definite $T=\sqrt{D^tD}$, and with orthogonal $U \in O(N)$. If $D$ is invertible, then $U$ is uniquely determined and we write $U = Pol(D)$.
\end{proposition}

The polar decomposition can be deduced from the singular value decomposition (again, see \cite{hjo} for details):

\begin{proposition}
If $D=V\Delta W^t$ with $V,W$ orthogonal and $\Delta$ diagonal is the singular values decomposition of $D$, then $Pol(D)=VW^t$.
\end{proposition}

Let us discuss now the notion of almost Hadamard matrix, from \cite{bcs}, \cite{bne}, \cite{bns}, \cite{bnz}. Recall first that the coordinate $1$-norm of a matrix $M\in M_N(\mathbb C)$ is given by: 
$$||M||_1=\sum_{ij}|M_{ij}|$$

The importance of the coordinate $1$-norm in relation with Hadamard matrices was realized in \cite{bcs}, where the following observation was made:

\begin{proposition}\label{prop:CS-AHM}
For $U\in O(N)$ we have $||U||_1\leq N\sqrt{N}$, with equality if and only if the rescaled matrix $H=\sqrt{N}U$ is Hadamard.
\end{proposition}

\begin{proof}
This follows from $||U||_2=\sqrt{N}$ by Cauchy-Schwarz, see \cite{bcs}.
\end{proof}

Since the global maximum of the $1$-norm over the orthogonal group is quite difficult to find, in \cite{bns}, \cite{bnz} the {\em local maxima} of the $1$-norm were introduced:

\begin{definition}
An {\em almost Hadamard matrix} (AHM) is a square matrix $H\in M_N(\mathbb R)$ having the property that $U=H/\sqrt{N}$ is a local maximum of the $1$-norm on $O(N)$.
\end{definition}

According to Proposition \ref{prop:CS-AHM}, these matrices can be thought of as being generalizations of the Hadamard matrices. Here is a basic example, which works at any $N\geq 3$:
$$K_N=\frac{1}{\sqrt{N}}\begin{bmatrix}
2-N&2&\ldots&2\\
2&2-N&\ldots&2\\
\ldots&\ldots&\ldots&\ldots\\
2&2&\ldots&2-N
\end{bmatrix}$$

We have the following characterization of rescaled AHM, from \cite{bns}:

\begin{theorem}
A matrix $U\in O(N)$ locally maximizes the $1$-norm on $O(N)$ if and only if $U_{ij}\neq 0$, and $U^tS\geq 0$, where $S_{ij}=sgn(U_{ij})$.
\end{theorem}

\begin{proof}
This follows from basic differential geometry, with $U_{ij}\neq 0$ coming from a rotation trick, and $U^tS\geq 0$ being the Hessian of the 1-norm around $U$. See \cite{bns}.
\end{proof}

The above proof shows that $U$ is a strict maximizer of the 1-norm, in the sense that we have $||U||_1>||U_\varepsilon||_1$ for $U_\varepsilon\neq U$ close to $U$, when $U^tS>0$. This is important for us, because in what follows we will sometimes need $S$ to be invertible.

In what follows we will be precisely interested in the sign matrices $S$:

\begin{definition}
A matrix $S\in M_N(\pm1)$ is called an {\em almost Hadamard sign pattern} (AHP) if there exists an almost Hadamard matrix $H\in M_N(\mathbb R)$ such that $S_{ij}=sgn(H_{ij})$.
\end{definition}

Here ``P'' comes at the same time from ``pattern'' and ``phase''.

Note that if a sign matrix $S$ is an AHP, then there exists a {\em unique} almost Hadamard matrix $H$ such that $S_{ij}=sgn(H_{ij})$, namely $H =\sqrt{N}Pol(S)$. Since the polar part is not uniquely defined for singular sign matrices, in what follows, we shall mostly be concerned with invertible AHPs and AHMs. 

\section{Submatrices of Hadamard matrices}\label{sec:submatrices}

In this section, we start analyzing square {\em submatrices} of Hadamard matrices. By permuting rows and columns, we can always reduce to the following situation:

\begin{definition}
$D\in M_d(\pm1)$ is called a submatrix of $H\in M_N(\pm1)$ if we have
$$H=\begin{bmatrix}A&B\\C&D\end{bmatrix}$$
up to a permutation of the rows and columns of $H$. We set $r=size(A)=N-d$.
\end{definition}

We recall that the $n$-th Walsh matrix is $W_N=[^+_+{\ }^+_-]^{\otimes n}$, with $N=2^n$. Here, and in what follows, we use the tensor product convention $(H\otimes K)_{ia,jb}=H_{ij}K_{ab}$, with the lexicographic order on the double indices. Here are the first 3 such matrices:
$$W_2=\begin{bmatrix}+&+\\+&-\end{bmatrix},\quad
W_4=\begin{bmatrix}+&+&+&+\\+&-&+&-\\+&+&-&-\\+&-&-&+\end{bmatrix},\quad 
W_8=
\begin{bmatrix}
+&+&+&+&+&+&+&+\\
+&-&+&-&+&-&+&-\\
+&+&-&-&+&+&-&-\\
+&-&-&+&+&-&-&+\\
+&+&+&+&-&-&-&-\\
+&-&+&-&-&+&-&+\\
+&+&-&-&-&-&+&+\\
+&-&-&+&-&+&+&-
\end{bmatrix}$$

Observe that any $D\in M_2(\pm1)$ having distinct columns appears as a submatrix of $W_4$, and that any $D\in M_2(\pm1)$ appears as a submatrix of $W_8$. In fact, we have:

\begin{proposition}
Let $D\in M_d(\pm1)$ be an arbitrary sign matrix.
\begin{enumerate}
\item If $D$ has distinct columns, then $D$ is as submatrix of $W_N$, with $N=2^d$.

\item In general, $D$ appears as submatrix of $W_M$, with $M=2^{d+ \lceil \log_2 d \rceil}$.
\end{enumerate}
\end{proposition}

\begin{proof}
(1) Set $N=2^d$. If we use length $d$ bit strings $x,y\in\{0,1\}^d$ as indices, then:
$$(W_N)_{xy}=(-1)^{\sum x_iy_i}$$

Let $\widetilde{W}_N\in M_{d\times N}(\pm1)$ be the submatrix of $W_N$ having as row indices the strings of type $x_i=(\underbrace{0\ldots 0}_i\,1\,\underbrace{0\ldots0}_{N-i-1})$. Then for $i\in\{1,\ldots,d\}$ and $y\in\{0,1\}^d$, we have:
\vskip-3mm
$$(\widetilde{W}_N)_{iy}=(-1)^{y_i}$$

Thus the columns of $\widetilde{W}_N$ are the $N$ elements of $\{\pm 1\}^d$, which gives the result. 

(2) Set $R=2^{\lceil \log_2 d \rceil} \geq d$. Since the first row of $W_R$ contains only $1$s, $W_R\otimes W_N$ contains as a submatrix $R$ copies of $\widetilde{W}_N$, in which $D$ can be embedded, finishing the proof.
\end{proof}

Let us go back now to Definition 2.1, and try to relate the matrices $A,D$ appearing there. The following result, due to Sz\"{o}ll\H{o}si \cite{szo}, is a first one in this direction:

\begin{theorem}\label{thm:szollosi}
If $U=\begin{bmatrix}A&B\\C&D\end{bmatrix}$ is unitary, with $A\in M_r(\mathbb C)$, $D\in M_d(\mathbb C)$, then:
\begin{enumerate}
\item The singular values of $A,D$ are identical, up to $|r-d|$ values of $1$. 

\item $\det A=\det U\cdot\overline{\det D}$, so in particular, $|\det A|=|\det D|$.
\end{enumerate}
\end{theorem}

\begin{proof}
Here is a simplified proof. From the unitarity of $U$, we get:
\begin{align*}
A^*A+C^*C&=I_r\\
CC^*+DD^*&=I_d\\
AC^*+BD^*&=0_{r\times d}
\end{align*}

(1) This follows from the first two equations, and from the well-known fact that the matrices $CC^*,C^*C$ have the same eigenvalues, up to $|r-d|$ values of $0$.

(2) By using the above unitarity equations, we have:
$$\begin{bmatrix}A&0\\C&I\end{bmatrix}
=\begin{bmatrix}A&B\\C&D\end{bmatrix}
\begin{bmatrix}I&C^*\\0&D^*\end{bmatrix}$$

The result follows by taking determinants.
\end{proof}

\section{Polar parts, norm estimates}\label{sec:polar}

In this section we state and prove our main results. Our first goal is to find a formula for the polar decomposition of $D$. Let us introduce:

\begin{definition}\label{def:XY-A}
Associated to any $A\in M_r(\pm1)$ are the matrices
\begin{eqnarray*}
X_A&=&(\sqrt{N}I_r+\sqrt{A^tA})^{-1}Pol(A)^t\\
Y_A&=&(\sqrt{N}I_r+\sqrt{AA^t})^{-1}
\end{eqnarray*}
depending on a parameter $N$.
\end{definition}

Observe that, in terms of the polar decomposition $A=VP$, we have:
\begin{eqnarray*}
X_A&=&(\sqrt{N}+P)^{-1}V^t\\
Y_A&=&V(\sqrt{N}+P)^{-1}V^t
\end{eqnarray*}

The idea now is that, under the general assumptions of Theorem 2.3, the polar parts of $A,D$ are related by a simple formula, with the passage $Pol(A)\to Pol(D)$ involving the above matrices $X_A,Y_A$. In what follows we will focus on the case that we are interested in, namely with $U\in U(N)$ replaced by $U=\sqrt{N}H$ with $H\in M_N(\pm1)$ Hadamard.

In the non-singular case, we have the following lemma:

\begin{lemma}\label{lem:polar-formula}
If $H=\begin{bmatrix}A&B\\C & D\end{bmatrix}\in M_N(\pm1)$ is Hadamard, with $A\in M_r(\pm1)$ invertible, $D\in M_d(\pm1)$, and $\|A\| < \sqrt N$, the polar decomposition $D=UT$ is given by $$U=\frac{1}{\sqrt{N}}(D-E)\qquad T=\sqrt{N}I_d-S$$ 
with $E=CX_AB$ and $S=B^tY_AB$.
\end{lemma}

\begin{proof}
Since $H$ is Hadamard, we can use the formulae coming from:
$$\begin{bmatrix}A&B\\C&D\end{bmatrix}\begin{bmatrix}A^t&C^t\\B^t&D^t\end{bmatrix}=\begin{bmatrix}A^t&C^t\\B^t&D^t\end{bmatrix}\begin{bmatrix}A&B\\C&D\end{bmatrix}=\begin{bmatrix}N&0\\0&N\end{bmatrix}$$

We start from the singular value decomposition of $A$: 
$$A=Vdiag(s_i)X^t$$

Here $V,X \in O(r)$, $s_i\in(0,\|A\|]$. From $AA^t+BB^t = NI_r$ we get $BB^t = V diag(N-s_i^2)V^t$, so the singular value decomposition of $B$ is as follows, with $Y\in O(d)$:
$$B = V\begin{bmatrix}diag(\sqrt{N-s_i^2})&0_{r\times(d-r)}\end{bmatrix}Y^t$$

Similarly, from $A^tA+C^tC = I_r$, we infer the singular value decomposition for $C$, the result being that there exists an orthogonal matrix $\widetilde{Z} \in O(d)$ such that: 
$$C=-\widetilde Z\begin{bmatrix}diag(\sqrt{N-s_i^2})\\0_{(d-r)\times r}\end{bmatrix}X^t$$

From $B^tB+D^tD = NI_d$, we obtain: 
$$D^tD = Y (diag(s_i^2)\oplus N I_{(d-r)}) Y^t$$

Tus the polar decomposition of $D$ reads:
$$D = UY (diag(s_i)\oplus \sqrt N I_{(d-r)}) Y^t$$

Let $Z = UY$ and use the orthogonality relation $CA^t+DB^t=0_{d \times r}$ to obtain:
$$\widetilde Z \begin{bmatrix}diag(s_i\sqrt{N-s_i^2})\\0_{(d-r)\times r}\end{bmatrix} = Z \begin{bmatrix}diag(s_i\sqrt{N-s_i^2})\\0_{(d-r)\times r}\end{bmatrix}$$

From the hypothesis, we have $s_i\sqrt{N-s_i^2} > 0$ and thus  $Z^t\widetilde Z = I_r \oplus Q$, for some orthogonal matrix $Q \in O(d)$. Plugging $\widetilde Z = Z(I_r \oplus Q)$ in the singular value decomposition formula for $C$, we obtain:
$$C=-Z(I_r \oplus Q)\begin{bmatrix}diag(\sqrt{N-s_i^2})\\0_{(d-r)\times r}\end{bmatrix}X^t = -Z\begin{bmatrix}diag(\sqrt{N-s_i^2})\\0_{(d-r)\times r}\end{bmatrix}X^t$$

To summarize, we have found $V,X \in O(r)$ and $Y,Z \in O(d)$ such that:
\begin{align*}
A &=Vdiag(s_i)X^t\\
B &= V\begin{bmatrix}diag(\sqrt{N-s_i^2})&0_{r\times(d-r)}
\end{bmatrix}Y^t\\
C &=-Z\begin{bmatrix}diag(\sqrt{N-s_i^2})\\0_{(d-r)\times r}\end{bmatrix}X^t\\
D &= Z (diag(s_i)\oplus \sqrt N I_{(d-r)}) Y^t
\end{align*}

Now with $U,T,E,S$ defined as in the statement, we obtain:
\begin{eqnarray*}
U&=&ZY^t\\
E&=&Z(diag(\sqrt{N}-s_i)\oplus0_{d-r})Y^t\\
\sqrt{A^tA}&=&Xdiag(s_i)X^t\\
(\sqrt{N}I_r+\sqrt{A^tA})^{-1}&=&Xdiag(1/(\sqrt{N}+s_i))X^t\\
X_A&=&Xdiag(1/(\sqrt{N}+s_i))V^t\\
CX_AB&=&Z(diag(\sqrt{N}-s_i)\oplus0_{d-r})Y^t
\end{eqnarray*}

Thus we have $E=CX_AB$, as claimed. Also, we have:
\begin{eqnarray*}
T&=&Y(diag(s_i)\oplus\sqrt{N}I_{d-r})Y^t\\
S&=&Y(diag(\sqrt{N}-s_i)\oplus0_{d-r})Y^t\\
\sqrt{AA^t}&=&Vdiag(s_i)V^t\\
Y_A&=&Vdiag(1/(\sqrt{N}+s_i))V^t\\
B^tY_AB&=&Y(diag(\sqrt{N}-s_i)\oplus0_{d-r})Y^t
\end{eqnarray*}

Hence, $S=B^tY_AB$, as claimed, and we are done.
\end{proof}

Note that, in the above statement, when $r<\sqrt N$, the condition $\|A\| < \sqrt N$ is automatically satisfied. 

As a first application, let us try to find out when $D$ is AHP. For this purpose, we must estimate the quantity $||E||_\infty=\max_{ij}|E_{ij}|$:

\begin{lemma}\label{lem:bound-E}
Let $H=\begin{bmatrix}A&B\\C & D\end{bmatrix}\in M_N(\pm1)$ be a Hadamard matrix, with $A\in M_r(\pm1)$, $D\in M_d(\pm1)$ and $r\leq d$. Then, $Pol(D)=\frac{1}{\sqrt{N}}(D-E)$, with $E$ satisfying:
\begin{enumerate}
\item $||E||_\infty\leq\frac{r\sqrt{r}}{\sqrt{r}+\sqrt{N}}$ when $A$ is Hadamard.

\item $||E||_\infty\leq\frac{r^2c\sqrt{N}}{N-r^2}$ if $r^2<N$, with $c=||Pol(A)-\frac{A}{\sqrt{N}}||_\infty$.

\item $||E||_\infty\leq\frac{r^2(1+\sqrt{N})}{N-r^2}$ if $r^2<N$.
\end{enumerate}
\end{lemma}

\begin{proof}
We use the basic fact that for two matrices $X\in M_{p\times r}(\mathbb C),Y\in M_{r\times q}(\mathbb C)$ we have $||XY||_\infty\leq r||X||_\infty||Y||_\infty$. Thus, according to Lemma \ref{lem:polar-formula}, we have:
$$||E||_\infty=||CX_AB||_\infty\leq r^2||C||_\infty||X_A||_\infty||B||_\infty=r^2||X_A||_\infty$$

(1) If $A$ is Hadamard, $AA^t =rI_r$, $Pol(A)=A/\sqrt{r}$ and thus: 
$$X_A=(\sqrt{N}I_r+\sqrt{r}I_r)^{-1}\frac{A^t}{\sqrt{r}}=\frac{A^t}{r+\sqrt{rN}}$$

Thus $||X_A||_\infty=\frac{1}{r+\sqrt{rN}}$, which gives the result.

(2) According to the definition of $X_A$, we have:
\begin{eqnarray*}
X_A
&=&(\sqrt{N}I_r+\sqrt{A^tA})^{-1}Pol(A)^t\\
&=&(NI_r-A^tA)^{-1}(\sqrt{N}I_r-\sqrt{A^tA})Pol(A)^t\\
&=&(NI_r-A^tA)^{-1}(\sqrt{N}Pol(A)-A)^t
\end{eqnarray*}

We therefore obtain:
$$||X_A||_\infty\leq r||(NI_r-A^tA)^{-1}||_\infty||\sqrt{N}Pol(A)-A||_\infty
=\frac{rc}{\sqrt{N}}\Big|\Big|\left(I_r-\frac{A^tA}{N}\right)^{-1}\Big|\Big|_\infty$$

Now by using $||A^tA||_\infty\leq r$, we obtain:
$$\Big|\Big|\left(I_r-\frac{A^tA}{N}\right)^{-1}\Big|\Big|_\infty\leq\sum_{k=0}^\infty\frac{||(A^tA)^k||_\infty}{N^k}\leq\sum_{k=0}^\infty\frac{r^{2k-1}}{N^k}=\frac{1}{r}\cdot\frac{1}{1-r^2/N}=\frac{N}{rN-r^3}$$

Thus $||X_A||_\infty\leq \frac{rc}{\sqrt{N}}\cdot\frac{N}{rN-r^3}=\frac{c\sqrt{N}}{N-r^2}$, which gives the result.

(3) This follows from (2), because $c\leq||Pol(A)||_\infty+||A/\sqrt{N}||_\infty\leq 1+\frac{1}{\sqrt{N}}$.
\end{proof}

We can now state and prove our main result in this paper:

\begin{theorem}\label{thm:AHP}
Let $H=\begin{bmatrix}A&B\\C & D\end{bmatrix}$ be Hadamard, with $A\in M_r(\pm1),H\in M_N(\pm1)$.
\begin{enumerate}
\item If $A$ is Hadamard, and $N>r(r-1)^2$, then $D$ is AHP. 

\item If $N>\frac{r^2}{4}(x+\sqrt{x^2+4})^2$, where $x=r||Pol(A)-\frac{A}{\sqrt{N}}||_\infty$, then $D$ is AHP.

\item If $N>\frac{r^2}{4}(r+\sqrt{r^2+8})^2$, then $D$ is AHP.
\end{enumerate}
\end{theorem}

\begin{proof}
(1) This follows from Lemma \ref{lem:bound-E} (1), because:
$$\frac{r\sqrt{r}}{\sqrt{r}+\sqrt{N}}<1\iff r<1+\sqrt{N/r}\iff r(r-1)^2<N$$

(2) This follows from Lemma \ref{lem:bound-E} (2), because:
$$\frac{r^2c\sqrt{N}}{N-r^2}<1\iff N-r^2c\sqrt{N}>r^2\iff (2\sqrt{N}-r^2c)^2>r^4c^2+4r^2$$

Indeed, this is equivalent to $2\sqrt{N}>r^2c+r\sqrt{r^2c^2+4}=r(x+\sqrt{x^2+4})$, for $x=rc=r||Pol(A)-\frac{A}{\sqrt{N}}||_\infty$.

(3) This follows from Lemma \ref{lem:bound-E} (3), because:
$$\frac{r^2(1+\sqrt N)}{N-r^2}<1\iff N-r^2\sqrt{N}>2r^2\iff(2\sqrt{N}-r^2)^2>r^4+8r^2$$

Indeed, this is equivalent to $2\sqrt{N}>r^2+r\sqrt{r^2+8}$, which gives the result.
\end{proof}

As a technical comment, for $A\in M_r(\pm1)$ Hadamard, Lemma \ref{lem:bound-E} (2) gives:
$$||E||_\infty\leq\frac{r^2\sqrt{N}}{N-r^2}\left(\frac{1}{\sqrt{r}}-\frac{1}{\sqrt{N}}\right)=\frac{r\sqrt{r}N-r^2}{N-r^2}$$

Thus $||E||_\infty<1$ for $N>r^3$, which is slightly weaker than Theorem \ref{thm:AHP} (1). 

\section{Complements of small sign patterns}\label{sec:small-r}

In view of the results above, it is convenient to make the following convention:

\begin{definition}
We denote by $\{x\}_{m \times n} \in M_{m \times n}(\mathbb R)$ the all-$x$, $m \times n$ matrix, and by
$$\begin{Bmatrix}x_{11}&\ldots&x_{1l}\\ \ldots&\ldots&\ldots\\ x_{k1}&\ldots&x_{kl}\end{Bmatrix}_{(m_1,\ldots,m_k) \times (n_1, \ldots, n_l)}$$
the matrix having all-$x_{ij}$ rectangular blocks $X_{ij}=\{x_{ij}\}_{m_i \times n_j} \in M_{m_i \times n_j}(\mathbb R)$, of prescribed size. In the case of square diagonal blocks, we simply write $\{x\}_n= \{x\}_{n \times n}$ and 
$$\begin{Bmatrix}x_{11}&\ldots&x_{1k}\\ \ldots&\ldots&\ldots\\ x_{kk}&\ldots&x_{kk}\end{Bmatrix}_{n_1, \ldots n_k} = \begin{Bmatrix}x_{11}&\ldots&x_{1k}\\ \ldots&\ldots&\ldots\\ x_{k1}&\ldots&x_{kk}\end{Bmatrix}_{(n_1,\ldots,n_k) \times (n_1, \ldots, n_k)}$$
\end{definition}

Modulo equivalence, the $\pm1$ matrices of size $r=1,2$ are as follows:
$$\begin{bmatrix}+\end{bmatrix}_{(1)}\qquad
\begin{bmatrix}+&+\\+&-\end{bmatrix}_{(2)}\qquad
\begin{bmatrix}+&+\\+&+\end{bmatrix}_{(2')}$$

In the cases $(1)$ and $(2)$ above, where the matrix $A$ is invertible, the spectral properties of their complementary matrices are as follows:

\begin{theorem}\label{thm:r-12}
For the $N\times N$ Hadamard matrices of type
$$\begin{bmatrix}+&+\\ +&D\end{bmatrix}_{(1)}\qquad
\begin{bmatrix}
+&+&+&+\\
+&-&+&-\\
+&+&D_{00}&D_{01}\\
+&-&D_{10}&D_{11}
\end{bmatrix}_{(2)}
$$
the polar decomposition $D=UT$ with $U=\frac{1}{\sqrt{N}}(D-E)$, $T=\sqrt{N}I-S$ is given by:
$$E_{(1)}=\begin{Bmatrix}\frac{1}{1+\sqrt{N}}\end{Bmatrix}_{N-1}\qquad
E_{(2)}=\frac{2}{2+\sqrt{2N}}\begin{Bmatrix}1&1\\1&-1\end{Bmatrix}_{N/2-1,N/2-1}$$
$$S_{(1)}=\begin{Bmatrix}\frac{1}{1+\sqrt{N}}\end{Bmatrix}_{N-1}\qquad
S_{(2)}=\frac{2}{\sqrt{2}+\sqrt{N}}\begin{Bmatrix}1&0\\0&1\end{Bmatrix}_{N/2-1,N/2-1}$$
In particular, all the matrices $D$ above are AHP.
\end{theorem}

\begin{proof}
For $A\in M_r(\pm1)$ Hadamard, the quantities in Definition \ref{def:XY-A} are:
$$X_A=\frac{A^t}{r+\sqrt{rN}}\qquad Y_A=\frac{I_r}{\sqrt{r}+\sqrt{N}}$$

These formulae follow indeed from $AA^t=A^tA=rI_r$ and $Pol(A)=A/\sqrt{r}$.

(1) Using the notation introduced in Definition \ref{def:XY-A}, we have here $B_{(1)} = \{1\}_{1 \times N-1}$ and $C_{(1)} = B_{(1)}^t$. Since $A_{(1)}=[+]$ is Hadamard we have $X_{A_{(1)}}=Y_{A_{(1)}}=\frac{1}{1+\sqrt{N}}$, and so: 
\begin{eqnarray*}
E_{(1)}&=&\frac{1}{1+\sqrt{N}}\{1\}_{N-1 \times 1}[1]\{1\}_{1 \times N-1}=\frac{1}{1+\sqrt{N}}\{1\}_{N-1}\\
S_{(1)}&=&\frac{1}{1+\sqrt{N}}\{1\}_{N-1 \times 1}\{1\}_{1 \times N-1}=\frac{1}{1+\sqrt{N}}\{1\}_{N-1}
\end{eqnarray*}

(2) Using the orthogonality of the first two lines in $H_{(2)}$, we find that the matrices $D_{00}$ and $D_{11}$ have size $N/2-1$. Since $A_{(2)}=[^+_+{\ }^+_-]$ is Hadamard we have $X_{A_{(2)}}=\frac{A}{2+\sqrt{2N}}$, $Y_{A_{(2)}}=\frac{I_2}{\sqrt{2}+\sqrt{N}}$, and so:
\begin{eqnarray*}
E_{(2)}&=&\frac{1}{2+\sqrt{2N}}\begin{Bmatrix}1&1\\1&-1\end{Bmatrix}_{(N/2-1,N/2-1) \times (1,1)}\begin{bmatrix}1&1\\1&-1\end{bmatrix}\begin{Bmatrix}1&1\\1&-1\end{Bmatrix}_{(1,1) \times (N/2-1,N/2-1)}\\
&=&\frac{2}{2+\sqrt{2N}}\begin{Bmatrix}1&1\\1&-1\end{Bmatrix}_{N/2-1,N/2-1}\\
S_{(2)}&=&\frac{1}{\sqrt{2}+\sqrt{N}}\begin{Bmatrix}1&1\\1&-1\end{Bmatrix}_{(N/2-1,N/2-1) \times (1,1)}\begin{Bmatrix}1&1\\1&-1\end{Bmatrix}_{(1,1) \times (N/2-1,N/2-1)}\\
&=&\frac{2}{\sqrt{2}+\sqrt{N}}\begin{Bmatrix}1&0\\0&1\end{Bmatrix}_{N/2-1,N/2-1}
\end{eqnarray*}
\end{proof}

As an illustration for the above computations, let us first work out the case $r=1,N=2$, with $H=[^+_+{\ }^+_-]$ being the first Walsh matrix. Here we have:
$$E=S=\frac{1}{1+\sqrt{2}}\implies U=-1,T=1$$

At $r=2,N=4$, consider the second Walsh matrix, written as in Theorem \ref{thm:r-12}:
$$W_4'=\begin{bmatrix}
+&+&+&+\\
+&-&+&-\\
+&+&-&-\\
+&+&-&+
\end{bmatrix}$$

We obtain the polar decomposition $D=UT$ of the corner $D=[^-_-{\ }^-_+]$:
$$E=\frac{1}{1+\sqrt{2}}\begin{bmatrix}1&1\\1&-1\end{bmatrix},
S=\frac{2}{2+\sqrt{2}}\begin{bmatrix}1&0\\0&1\end{bmatrix}\implies 
U=-\frac{1}{\sqrt{2}}\begin{bmatrix}1&1\\1&-1\end{bmatrix},
T=\sqrt{2}\begin{bmatrix}1&0\\0&1\end{bmatrix}$$

Let us record as well the following consequence of Theorem \ref{thm:r-12}:

\begin{corollary}
We have the formulae
\begin{eqnarray*}
\det(\lambda-T_{(1)})&=&(\lambda-1)(\lambda-\sqrt{N})^{4N-2}\\
\det(\lambda-T_{(2)})&=&(\lambda-\sqrt{2})^2(\lambda-\sqrt{N})^{N-4}
\end{eqnarray*}
so in particular $|\det D_{(1)}|=N^{N/2-1}$, $|\det D_{(2)}|=2N^{N/2-2}$.
\end{corollary}

\begin{proof}
As already noted in \cite{szo}, these formulae from \cite{km1}, \cite{km2} follow from Theorem 2.3. They are even more clear from the above formulae of $S_{(1)}$  and $S_{(2)}$, using  the fact that the non-zero eigenvalues of $\{^1_0{\ }^0_1\}_{a,b}$ are 
$a$ and $b$.
\end{proof}

Modulo equivalence, the $\pm1$ matrices of size $r=3$ are as follows:
$$\begin{bmatrix}+&+&+\\+&-&+\\+&+&-\end{bmatrix}_{(3)}\qquad
\begin{bmatrix}+&+&+\\+&+&+\\+&+&-\end{bmatrix}_{(3')}\qquad
\begin{bmatrix}+&+&+\\+&+&+\\+&+&+\end{bmatrix}_{(3'')}$$

Among those, only $(3)$ is invertible, here is the result:

\begin{proposition}
For the $N\times N$ Hadamard matrices of type
$$\begin{bmatrix}
+&+&+&+&+&+&+\\
+&-&+&+&+&-&-\\
+&+&-&+&-&+&-\\
+&+&+&D_{00}&D_{01}&D_{02}&D_{03}\\
+&+&-&D_{10}&D_{11}&D_{12}&D_{13}\\
+&-&+&D_{20}&D_{21}&D_{22}&D_{23}\\
+&-&-&D_{30}&D_{31}&D_{32}&D_{33}
\end{bmatrix}_{(3)}$$
the polar decomposition $D=UT$ with $U=\frac{1}{\sqrt{N}}(D-E)$, $T=\sqrt{N}I-S$ is given by:
$$E_{(3)}=\frac{1}{\sqrt N +1} \begin{Bmatrix}
x & y & y & 1\\
y & -y & x & -1\\
y & x & -y & -1\\
1 & -1 & -1 & -3
\end{Bmatrix}_{N/4-1,N/4-1,N/4-1,N/4}$$
$$S_{(3)}=\frac{1}{\sqrt N +1} \begin{Bmatrix}
z & t & t & -1\\
t & z & -t & 1\\
t & -t & z & 1\\
-1 & 1 & 1 & 3\end{Bmatrix}_{N/4-1,N/4-1,N/4-1,N/4}$$
where 
$$x=\frac{7 \sqrt N + 6}{3\sqrt N+6}, \quad y=\frac{5 \sqrt N + 6}{3\sqrt N+6}, \quad z=\frac{9 \sqrt N + 10}{3\sqrt N+6} , \quad t=\frac{3 \sqrt N + 2}{3\sqrt N+6}$$
In particular, if $N > \|A\|^2=4$, $D$ is an AHP.
\end{proposition}

\begin{proof}
By direct computation, we have
$$X_A = \frac{1}{3(\sqrt N +1)(\sqrt N +2)}\begin{bmatrix}
\sqrt N & 2\sqrt N + 3 & 2\sqrt N + 3 \\
2\sqrt N + 3 & -(2\sqrt N + 3) & \sqrt N \\
2\sqrt N + 3 & \sqrt N & -(2\sqrt N + 3)
\end{bmatrix}$$
From the orthogonality condition for the first three lines of $H_{(3)}$, we find that the dimensions of the matrices $D_{00},D_{11},D_{22}$ and $D_{33}$ are, respectively, $N/4-1,N/4-1,N/4-1$ and $N/4$. Thus, we have
$$B_{(3)} = \begin{Bmatrix}
1 & 1 & 1 & 1 \\
1 & 1 & -1 & -1 \\
1 & -1 & 1 & -1
\end{Bmatrix}_{(1,1,1) \times (N/4-1,N/4-1,N/4-1,N/4)}$$
and $C_{(3)} = B_{(3)}^t$. The formulas for $E_{(3)}$ and $S_{(3)}$ follow by direct computation.

Since 
$$3>\frac{7 \sqrt N + 6}{3\sqrt N+6}>\frac{5 \sqrt N + 6}{3\sqrt N+6}>1,$$
we have $\|E_{(3)}\|_\infty = 3/(\sqrt N+1)$ and the conclusion about $D$ being AHP follows.
\end{proof}

Note that in the case $N=4$, there is a unique way to complement the matrix $A_{(3)}$ above into a $4 \times 4$ Hadamard matrix. Since the complement in this case is simply $D=[1]$, we conclude that, for all $N \geq 4$, the complement of $A_{(3)}$ inside a $N \times N$ Hadamard matrix is AHP.

In this case, using Theorem \ref{thm:szollosi}, we get $\det(\lambda-T)=(\lambda+1)(\lambda-2)^2(\lambda-1/\sqrt{N})^{N-3}$, and so $|\det D|=2N^{(N-3)/2}$. 

\section{Examples of non-AHP sign patterns}\label{sec:non-AHP}

In the previous section, we have shown that for $r=1,2,3$, all invertible $r \times r$ sign patterns are complemented by AHP matrices inside Hadamard matrices. In the following proposition we show that this is not the case for larger values of $r$. Recall that for a matrix $D$ to be AHP, it must be invertible, its polar part $U=Pol(D)$ must have non-zero entries, and $D = sgn(U)$ must hold.

\begin{proposition}
Consider the Walsh matrix $W_8$, and the Paley matrix $H_{12}$.
\begin{enumerate}
\item $W_8$ has a $4\times 4$ submatrix which is not AHP, due to a $U_{ij}=0$ reason.

\item $H_{12}$ has a $7\times7$ submatrix which is not AHP, due to a $D_{ij}=-1,U_{ij}>0$ reason.
\end{enumerate}
\end{proposition}

\begin{proof}
(1) Let $A$ be the submatrix of $W_8$ (see section \ref{sec:submatrices}) having the rows and columns with indices $1,2,3,5$. Then $A,D$ and $U=Pol(D)$ are as follows:
$$A=\begin{bmatrix}
+&+&+&+\\
+&-&+&+\\
+&+&-&+\\
+&+&+&-
\end{bmatrix},\quad
D=\begin{bmatrix}
+&-&-&+\\
-&+&-&+\\
-&-&+&+\\
+&+&+&-
\end{bmatrix},\quad
U=\begin{bmatrix}
\frac{2}{3}&-\frac{1}{3}&-\frac{1}{3}&\frac{1}{\sqrt 3}\\
-\frac{1}{3}&\frac{2}{3}&-\frac{1}{3}&\frac{1}{\sqrt 3}\\
-\frac{1}{3}&-\frac{1}{3}&\frac{2}{3}&\frac{1}{\sqrt 3}\\
\frac{1}{\sqrt 3}&\frac{1}{\sqrt 3}&\frac{1}{\sqrt 3}&0 
\end{bmatrix}$$

(2) Consider indeed the unique $12\times 12$ Hadamard matrix, written as:
$$H_{12}=\begin{bmatrix}
+&-&-&-&-&-&-&-&-&-&-&-\\
+&+&-&+&-&-&-&+&+&+&-&+\\
+&+&+&-&+&-&-&-&+&+&+&-\\
+&-&+&+&-&+&-&-&-&+&+&+\\
+&+&-&+&+&-&+&-&-&-&+&+\\
+&+&+&-&+&+&-&+&-&-&-&+\\
+&+&+&+&-&+&+&-&+&-&-&-\\
+&-&+&+&+&-&+&+&-&+&-&-\\
+&-&-&+&+&+&-&+&+&-&+&-\\
+&-&-&-&+&+&+&-&+&+&-&+\\
+&+&-&-&-&+&+&+&-&+&+&-\\
+&-&+&-&-&-&+&+&+&-&+&+
\end{bmatrix}$$
        
Let $A$ be the submatrix having the rows and columns with indices $1,2,3,5,6$:
$$A=\begin{bmatrix}
+&-&-&-&-\\
+&+&-&-&-\\
+&+&+&+&-\\
+&+&-&+&-\\
+&+&+&+&+
\end{bmatrix},\quad
D=\begin{bmatrix}
+&-&-&-&+&+&+\\
+&+&-&+&-&-&-\\
+&+&+&-&+&-&-\\
+&-&+&+&-&+&-\\
-&+&-&+&+&-&+\\
-&+&+&-&+&+&-\\
-&+&+&+&-&+&+
\end{bmatrix}$$

Then $D$ is invertible, and its polar part is given by:
$$U\approx\begin{bmatrix}
0.51&-0.07&-0.37&-0.22&0.35&0.51&0.37\\
0.37&0.51&-0.51&0.22&-0.35&-0.07&-0.37\\
0.51&0.37&0.51&-0.22&0.35&-0.37&-0.07\\
0.35&-0.35&0.35&0.61&0.03&0.35&-0.35\\
-0.22&0.22&-0.22&0.61&0.61&-0.22&0.22\\
-0.37&0.37&0.07&-0.22&0.35&0.51&-0.51\\
-0.07&0.51&0.37&0.22&-0.35&0.37&0.51 
\end{bmatrix}$$

Now since $D_{45}=-1$ and $U_{45}\approx 0.03>0$, this gives the result.
\end{proof}

\end{document}